\documentclass[12pt]{amsart}
\usepackage{amsmath,amsthm,
amsopn,
amssymb}

\newtheorem{theorem}{Theorem}[section]
\newtheorem{lemma}[theorem]{Lemma}
\newtheorem{cor}[theorem]{Corollary}
\newtheorem{prop}[theorem]{Proposition}

\theoremstyle{definition}
\newtheorem{definition}[theorem]{Definition}
\newtheorem{definitions}[theorem]{Definitions}
\newtheorem{example}[theorem]{Example}
\newtheorem{examples}[theorem]{Examples}
\newtheorem{notn}[theorem]{Notation}

\theoremstyle{remark}
\newtheorem{remark}[theorem]{Remark}

\numberwithin{equation}{section}

\DeclareMathOperator{\ham}{ham}
\DeclareMathOperator{\Pz}{PZ}
\DeclareMathOperator{\curl}{curl}
\DeclareMathOperator{\grad}{grad}
\DeclareMathOperator{\Pspec}{P.Spec}
\DeclareMathOperator{\spec}{Spec}
\DeclareMathOperator{\Jac}{Jac}
\newcommand{\ov}{\overline}

\newcommand{\C}{{\mathbb C}}
\newcommand{\R}{{\mathbb R}}
\newcommand{\N}{{\mathbb N}}
\newcommand{\Q}{{\mathbb Q}}

\begin{document}
\bibliographystyle{amsalpha}
\author{David A. Jordan}

\address
{School of Mathematics and Statistics\\
University of  Sheffield\\
Hicks Building\\
Sheffield S3~7RH\\
UK}

\email{d.a.jordan@sheffield.ac.uk}

\title{
Ore Extensions and Poisson algebras}



\begin{abstract}
For a derivation $\delta$ of a commutative Noetherian $\C$-algebra $A$, a homeomorphism is established between the prime spectrum of the Ore extension $A[z;\delta]$ and the Poisson prime spectrum of the polynomial algebra $A[z]$ endowed with the Poisson bracket such that $\{A,A\}=0$ and
$\{z,a\}=\delta(a)$ for all $a\in A$.
\end{abstract}

\maketitle
\section{Introduction}
The best known example of a simple Poisson algebra is the coordinate ring of the symplectic plane, that is the polynomial algebra $\C[z,x]$ with $\{z,x\}=1$. This corresponds to the best known example of a simple Ore extension $A[z;\delta]$, namely the Weyl algebra
$A_1(\C)$, generated by $x$ and $z$ subject to the relation $zx-xz=1$. Here $A=\C[x]$ and $\delta=d/dx$. The first known example of a Poisson bracket on $\C[x,y,z]$ for which $\C[x,y,z]$ is a simple Poisson algebra, due to Farkas \cite[Example following Lemma 15]{farkas2}, is such that $\{x,y\}=0$
and the hamiltonian $\{z,-\}$ acts on $\C[x,y]$ as the derivation $\delta=\partial_x+(1-xy)\partial_y$, where $\partial_x$ and $\partial_y$ are the partial derivatives. In the first known example, due to Bergman, see \cite{Cout2}, of a derivation $\delta$ for which the Ore extension $\C[x,y][z;\delta]$ is simple the derivation $\delta$ is $\partial_x+(1+xy)\partial_y$. The proofs of simplicity in \cite{farkas2} and \cite{Cout2} both remain valid  for the common generalization where $\delta=\partial_x+(1+\lambda xy)\partial_y$ for some $\lambda\in \C^*$, giving rise to corresponding families of simple Poisson algebras
and simple Ore extensions.
Unlike the case of the symplectic plane and the Weyl algebra, this correspondence does not appear to have been noted. These examples of simple Poisson algebras with corresponding simple Ore extensions are special cases of a general situation. Given any non-zero derivation $\delta$ of a commutative $\C$-algebra $A$, there is a Poisson bracket on the polynomial algebra $A[z]$ such that $\{A,A\}=0$ and
$\{z,a\}=\delta(a)$ for all $a\in A$. We shall show that if $A$ is Noetherian then the Poisson prime spectrum of $A[z]$ is homeomorphic to the prime spectrum of $A[z;\delta]$.
This fits into the philosophy of \cite{goodsemiclass}, in that $A[z]$ is the commutative fibre version of the semiclassical limit of the family
of noncommutative algebras $R_\alpha:=A[h][z;h\delta]/(h-\alpha)A[h][z;h\delta]$, where $\alpha \in\C^*$ and the derivation $\delta$ is extended to the polynomial algebra $A[h]$ by setting $\delta(h)=0$. Note that $R_\alpha\simeq A[z;\alpha\delta]$.

 In addition to Bergman's example, there are many known examples of simple derivations of $\C[x,y]$, for example see \cite{Cout,Cout2,Cout3,havran,dajox,nowetal,nowicki}. All such examples give rise to Poisson brackets for which $\C[x,y,z]$ is a simple Poisson algebra. In \cite{gwpaper1}, Goodearl and Warfield illustrated their study of Krull dimension in Ore extensions with
 some non-simple Ore extensions of $\C[x,y]$ with interesting prime spectra. In the final section we shall transfer these and some other known examples to the Poisson setting  and also answer a question from \cite{gwpaper1} on Ore extensions by constructing an accessible example of a derivation of $\C[x,y]$ giving rise to a Poisson bracket on $B:=\C[x,y,z]$ for which the height two prime ideal $yB+zB$ is Poisson but no height one prime ideal is Poisson.

\section{Background on Poisson algebras}
Our base field will always be $\C$ though the results are valid over any field of characteristic $0$.
In Remark \ref{fds} algebraic closure is pertinent. We denote the prime spectrum of a not-necessarily commutative ring  by $\spec R$.

\begin{definition}
A \emph{Poisson algebra} is $\C$-algebra $A$ with a Poisson bracket, that is a bilinear product
$\{-,-\}:A\times A\rightarrow A$ such that $A$ is a Lie algebra
under $\{-,-\}$ and, for all $a\in A$, the \emph{hamiltonian} $\ham(a):=\{a,-\}$ is a
$\C$-derivation of $A$.
\end{definition}

The following definitions and the claims made for them are well-known. One comprehensive reference is
\cite[Lemma 1.1 and thereabouts]{gooddm}.

\begin{definitions}\label{Delta}
Let $\Delta$ be a set of derivations of a commutative $\C$-algebra $A$. The $\Delta$-\emph{centre},
$Z_\Delta(A)$, of $A$ is $\{a\in A: \delta(a)=0\mbox{ for
all }\delta\in \Delta\}$.

An ideal $I$ of $A$ is a $\Delta$\emph{-ideal} if $\delta(I) \subseteq I$ for all $\delta\in \Delta$ and  a $\Delta$-ideal $P$ of $A$ is $\Delta$-\emph{prime} if, for all $\Delta$-ideals $I$ and $J$ of $A$,
$IJ\subseteq P$ implies $I\subseteq P$ or $J\subseteq P$. If $\Delta=\{\delta\}$ is a singleton then, in these and subsequent definitions, we replace $\Delta$ by $\delta$
rather than
 by $\{\delta\}$.

To say that $A$
is $\Delta$-\emph{simple}
 means that $0$ is the only proper $\Delta$-ideal $I$ of $A$. A derivation $\delta$ of $A$ is said to be {\it simple} if $A$ is $\delta$-simple.

The $\Delta$-\emph{core} of an ideal $I$ of $A$, denoted $(I:\Delta)$, is the largest $\Delta$-ideal of $A$ contained in $I$. If $P$ is a prime ideal of $A$ then
 $(P:\Delta)$ is prime, see \cite[Lemma 1.1(a)]{gooddm}.

If $I$ is a $\Delta$-ideal of $A$
then each derivation $\delta\in \Delta$ induces a derivation $\ov{\delta}$ of $A/I$ such that
$\ov{\delta}(a+I)=\delta(a)+I$ for all $a\in A$. If $I$ is a $\Delta$-ideal and is also prime then
$\ov\delta$ extends to the quotient field $Q(R/I)$ by the quotient rule, $\ov\delta(as^{-1})=(\ov s\ov\delta(\ov a)-\ov a\ov\delta(\ov s))\ov s^{-2}$.

A $\Delta$-ideal $P$ of $A$ is $\Delta$-\emph{primitive} if $P=(M:\Delta)$ for some maximal ideal $M$ of $A$.
Every $\Delta$-primitive ideal is $\Delta$-prime.

If $A$ is a Poisson algebra and $\Delta=\{\ham(b):b\in A\}$ then we replace the prefix $\Delta$- by the word Poisson. In particular an ideal $I$ of a Poisson algebra is a \emph{Poisson ideal} if $\{i,a\}\in I$ for all $a\in A$ and $i\in I$ and  $A$ is a \emph{simple Poisson algebra} if and only if the only Poisson ideals of $A$ are $0$ and $A$. The Poisson centre of $A$ and the Poisson core of a Poisson ideal $I$ of $A$ will be denoted by $\Pz(A)$  and
$\mathcal{P}(I)$ respectively.

\end{definitions}

\begin{lemma}\label{dpp}
Let $A$ be a commutative Noetherian $\C$-algebra and
 let $\Delta$ be a set of derivations of $A$. If $P$ is a
$\Delta$-prime ideal of $A$ then $P$ is prime.
\end{lemma}
\begin{proof}
See \cite[Lemma 1.1 (d)]{gooddm}.
\end{proof}

\begin{definitions} Let $\Delta$ be a set of derivations of a commutative Noetherian $\C$-algebra $A$.
The $\Delta$-prime spectrum of $A$, denoted $\Delta$-$\spec(A)$, is
the set of all $\Delta$-prime ideals of $A$ with the topology induced from the Zariski topology in $\spec(A)$.
The Poisson spectrum of $A$ will be denoted by $\Pspec(A)$. Thus a closed set in $\Pspec(A)$ has the form $V(I):=\{P\in \Pspec(A):P\supseteq I\}$ for some ideal $I$ of $A$. As is observed in \cite[\S6.1]{goodsemiclass},
replacing $I$ by the Poisson ideal that it generates, $I$ can be assumed to be a Poisson ideal.
\end{definitions}

\begin{definition}Let $A$ be a Poisson algebra and $I$ be a Poisson ideal of $A$. If the induced Poisson bracket on $A/I$ is zero, we say that
$I$ is \emph{residually
null}. This is equivalent to saying that $I$ contains
all elements of
the form $\{a,b\}$ where $a,b\in A$, or that $I$ contains all such
elements where $a,b\in G$ for some generating set $G$ for $A$.
The set of residually null Poisson prime ideals of $A$ is clearly closed in $\Pspec(A)$.
\end{definition}
\begin{definitions}
By a \emph{Poisson maximal ideal} we mean a maximal ideal
that is also Poisson whereas by a \emph{maximal Poisson ideal} we
mean a Poisson ideal that is maximal in the lattice of Poisson
ideals. These notions are not equivalent. Any Poisson maximal ideal is maximal Poisson  but the
converse is false as can be seen by considering the ideal $0$ in any  simple Poisson algebra that is not simple as an associative algebra, such as
$\C[y,z]$ with $\{y,z\}=1$.
\end{definitions}

\begin{definitions}
A G-\emph{domain} is a commutative integral domain $R$ such that the intersection of the non-zero prime ideals is non-zero, in other words $0$ is locally closed in $\spec R$. See \cite[Theorems 19 and 20 and the intermediate text]{kap}.
With $A$ and $\Delta$ as in Definitions~\ref{Delta}, let $P$ be a $\Delta$-prime ideal  of $A$. We shall say that $P$ is
$\Delta$-G if it is locally closed in $\Delta$-$\spec(A)$. To say that $A$ is $\Delta$-G means that $0$ is a $\Delta$-G ideal of $A$.

 If $P$ is a $\Delta$-ideal and prime, in particular if $A$ is Noetherian and $P$ is $\Delta$-prime, we say that $P$ is $\Delta$-\emph{rational} if $Z_{\ov{\Delta}}(Q(A/P))=\C$, where $\ov{\Delta}$ is the set of derivations of the quotient field $Q(A/P)$ induced, via $R/P$, by derivations belonging to $\Delta$.
\end{definitions}

\section{Semiclassical limits of Ore extensions}
\label{poissonore}
Let $A$ denote a commutative
$\C$-algebra that is also a domain and let $D$ be the polynomial algebra $A[h]$. Let $\delta$ be a derivation
of $A$ and extend $\delta$ to $D$ by setting $\delta(h)=0$. Then $h\delta$ is a derivation
of $D$ and we can form the Ore extension (or skew polynomial ring or ring of formal differential operators)
$T:=D[z;h\delta]$ in which elements have the form $\sum_0^n d_iz^i$, $d_i\in D$, and $zd-dz=h\delta(d)$ for all $d\in D$.
Note that $hz=zh$ and $h$ is a central non-unit regular element of $T$ such that $T/hT$ is isomorphic to
the commutative polynomial algebra $B:=A[z]$ and $T/(h-1)T$ is isomorphic to
the Ore extension $R:=A[z;\delta]$. If $\alpha\in \C^*$, then $T/(h-\alpha)T\simeq A[z;\alpha\delta]\simeq
A[z;\delta]$, where the final isomorphism maps $z$ to $\alpha z$. In this situation, there is
a well-defined Poisson bracket on $B$ such that
\[\{\ov u,\ov
v\}=\ov{h^{-1}[u,v]}\] for all $\ov u=u+hT$ and $\ov v=v+hT\in B$.
With this bracket, $B$  is the {\it semiclassical limit} of the family $A[z;\alpha\delta]$, $\alpha\in \C^*$, as in \cite[2.1]{goodsemiclass}, $T$ is a \emph{quantization} of the Poisson algebra $B$ in the sense of
\cite[Chapter III.5]{BGl} and $R$ is a \emph{deformation} of
$B$ in the sense of \cite{dajns}. A familiar example is obtained by taking $A=\C[x]$ and $\delta=d/dx$. Here
$R$ is the Weyl algebra $A_1(\C)$, with generators $x$ and $z$ subject to the relation $zx-xz=1$, and the semiclassical limit $B$ is $\C[x,z]$ with $\{z,x\}=1$, that is the coordinate ring of the symplectic plane.

To emphasise the role of the single derivation $\delta$, the Poisson bracket on $B$ will sometimes be written $\{-,-\}_\delta$. Thus
$\{a,b\}_\delta=0$  and $\{z,a\}_\delta=\delta(a)$ for all
$a,b\in A$. In the terminology of \cite{oh}, $B$ is a Poisson polynomial ring over $A$ for which the
Poisson bracket on $A$ and the derivation $\alpha$ are both zero.

\begin{lemma}\label{QB}
Let $A$ be a commutative $\C$-algebra with a derivation $\delta$ and let $B=A[z]$
equipped with the Poisson bracket $\{-,-\}_\delta$.
\begin{enumerate}
\item
For all $a,b\in A$ and all $m,n\in\N$, $\{az^m,bz^n\}=(ma\delta(b)-nb\delta(a))z^{m+n-1}.$
\item
Let $Q$ be a $\delta$-ideal of $A$. Then $QB$ is a Poisson ideal of $B$ and there is an isomorphism of Poisson algebras,
$\theta_Q:B/QB\rightarrow (A/Q)[z]$ given by \[\theta_Q\left(\left(\sum_{i=0}^n a_iz^i\right)+QB\right)=\sum_{i=0}^n (a_i+Q)z^i,\]
where the Poisson bracket on $A/Q[z]$ is $\{-,-\}_{\overline{\delta}}$.
\end{enumerate}
\end{lemma}
\begin{proof}
(i) is routine using the fact that the hamiltonians are derivations. The first statement in (ii) is immediate from (i) and the second statement is
straightforward.
\end{proof}

\begin{lemma}
\label{PcapA} Let $A$ be a commutative Noetherian $\C$-algebra that is also a domain, %
 let $\delta$ be a non-zero derivation of $A$ and let $P$ be a
non-zero Poisson prime ideal of $B:=C[z]$ for the Poisson bracket
$\{-,-\}_\delta$. Let $Q=P\cap A$.
\begin{enumerate}
\item
$Q$ is a non-zero $\delta$-prime ideal of $A$.
\item
If $\delta(A)\not\subseteq Q$ then $P=QB$.
\end{enumerate}
\end{lemma}
\begin{proof}
(i) Let $p=\sum_{i=0}^n a_iz^i$, with each $a_i\in A$, be an element of
minimal degree $n$ in $z$ among non-zero elements of $P$. Let $a\in A$ be such that $\delta(a)\neq 0$.
Then $(\ham
a)(p)=-\sum_{i=0}^n i\delta(a)a_iz^{i-1}\in P$. This
contradicts the minimality of $n$ unless $n=0$. Thus $n=0$ and
$Q\neq 0$.

As $P$ is a Poisson ideal of $B$, $Q$ is a $\delta$-ideal of $A$. Let $I$ and $J$ be $\delta$-ideals of $A$ such that $IJ\subseteq Q$. By Lemma~\ref{QB}(i), $IB$ and $JB$ are Poisson ideals of $B$.
Clearly $IBJB\subseteq P$ so either $IB\subseteq P$, whence $I\subseteq P\cap A$, or $JB\subseteq P$, whence $J\subseteq P\cap A$. Thus
$P\cap A$ is $\delta$-prime.

(ii) By Lemma~\ref{dpp}, $A/Q$ is a domain. Suppose that $\delta(A)\not\subseteq Q$, so that the induced Poisson bracket on the domain $A/Q$ is non-zero.
Clearly $QB\subseteq P$. If $QB\neq P$ then $\theta_Q(P/BQ)$ is a non-zero Poisson ideal of $(A/Q)[z]$ intersecting $A/Q$ in $0$. This is impossible by (i) applied to $A/Q$ so $QB=P$.
\end{proof}

The situation is similar for the prime spectrum of the Ore extension $R=C[z;\delta]$.
Let $A$ be a commutative $\C$-algebra with a derivation $\delta$ and let $Q$ be a $\delta$-ideal of $A$.
By \cite[\S1, final paragraph]{good},
 $QR$ is an ideal of $R$ and there is an isomorphism
$\psi_Q:R/QR\rightarrow A/Q[z;\overline{\delta}]$ given by \[\psi_Q\left(\left(\sum_{i=0}^n a_iz^i\right)+QR\right)=\sum_{i=0}^n (a_i+Q)z^i.\]

\begin{lemma}
\label{PcapAore} Let $A$ be a commutative $\C$-algebra that is also a domain,
 let $\delta$ be a non-zero derivation of $A$ and let $P$ be a
non-zero  prime ideal of $R:=A[z;\delta]$. Let $Q=P\cap A$.
\begin{enumerate}
\item
$Q$ is a non-zero $\delta$-prime ideal of $A$.
\item
If $\delta(A)\not\subseteq Q$ then $P=QR$.
\end{enumerate}
\end{lemma}
\begin{proof}
(i)  By
\cite[Lemma 1.3]{dajjac}, $Q$ is $\delta$-prime and, by \cite[Lemma 1]{primore}, $Q\neq 0$.

(ii) By Lemma~\ref{dpp} with $\Delta=\{\delta\}$, $A/Q$ is a domain. Suppose that $\delta(A)\nsubseteq Q$ so that the induced derivation $\overline\delta$ on the domain $A/Q$ is non-zero.
The ideal $QR$ is prime by \cite[Lemma 1.3]{dajjac} and $QR\subseteq P$. If $QB\neq P$ then $\phi_Q(P/RQ)$ is a non-zero ideal of $(A/Q)[z;\overline{\delta}]$ intersecting $A/Q$ in $0$. This is impossible by (i) applied to $A/Q$ so $QR=P$.
\end{proof}

\begin{cor}\label{simple}Let $A$ be a commutative $\C$-algebra that is a domain and let
 $\delta$ be a non-zero derivation of $A$. Let $R=A[z;\delta]$ and let $B$ be the Poisson algebra $A[z]$ with the Poisson bracket $\{-,-\}_\delta$. Then
$B$ is Poisson simple if and only if $R$ is simple if and only if $\delta$ is simple.
\end{cor}
\begin{proof}
It follows from Lemmas~\ref{PcapA} and \ref{PcapAore} that if $\delta$ is simple then $B$ is Poisson simple and $R$ is simple. On the other hand, if $J$ is a non-zero $\delta$-ideal of $A$ then $JR$ is a non-zero proper ideal of
$R$, by \cite[Lemma 1.3]{dajjac}, and $JB$ is a non-zero proper Poisson ideal of
$B$ by Lemma~\ref{QB}(ii).
\end{proof}

We now aim to generalize Corollary~\ref{simple} to establish a homeomorphism between $\spec R$ and $\Pspec B$. On each side, we shall partition the spectrum into two types of prime ideal.
\begin{notn}\label{twotypes} Let $A$ be a commutative $\C$-algebra and domain with a non-zero derivation $\delta$, let $R=A[z;\delta]$ and let $B=A[z]$
equipped with the Poisson bracket $\{-,-\}_\delta$.
Let $J=\delta(A)A$, which is a $\delta$-ideal of $A$, and let $S=(A/J)[z]$. Then
\begin{enumerate}
\item $JB=\{B,B\}B$ is a residually null Poisson ideal of $B$ and is contained in all residually null Poisson ideals of $B$.
\item  $\theta_J:B/JB\rightarrow S$ is an isomorphism of $\C$-algebras. The Poisson brackets are both $0$.
\item $JR$ is an ideal of $R$ such that $R/JR$ is commutative and $JR$ is contained in all ideals $I$ of $R$
such that $R/I$ is commutative.
\item $\psi_R:R/JR\rightarrow S$ is an isomorphism of commutative $\C$-algebras. The induced derivation $\overline{\delta}$ on $A/J$ is $0$.
\end{enumerate}
Let \[\Pspec_1(B)=\{P\in \Pspec B: P\text{ is residually null}\}=\{P\in \Pspec B:JB\subseteq P\}\] and let
$\Pspec_2(B)=\Pspec B\backslash\Pspec_1(B)$. By analogy, let \[\spec_1(R)=\{P\in \spec R: R/P\text{ is commutative}\}=\{P\in \spec R:JR\subseteq P\}\] and let
$\spec_2(R)=\spec R\backslash\spec_1(R)$.
Note that $\Pspec_1(B)$ and
$\spec_1(R)$ are closed in $\Pspec B$ and $\spec R$ respectively.
Also $\Pspec_1(B)$ is homeomorphic to $\spec(B/JB)$ and $\spec_1(R)$ is homeomorphic to $\spec(R/JR)$.

Let $\kappa$ be the isomorphism $\psi_J^{-1}\theta_J:B/JB\rightarrow R/JR$. Thus \[\kappa\left(\left(\sum_{i=0}^n a_iz^i\right)+JB\right)=
\left(\sum_{i=0}^n a_iz^i\right)+JR.\]
Then $\kappa$ induces a homeomorphism between $\spec(R/JR)$ and $\spec(B/JB)$ and there is a homeomorphism $\Gamma_1:\Pspec_1(B)\rightarrow \spec_1(R)$ such that $\Gamma_1(P)/JR=\psi(P/JB)$ for all
$P\in \Pspec_1(B)$.
\end{notn}

\begin{theorem}\label{homeo}Let $A$ be a Noetherian $\C$-algebra that is a domain and let
 $\delta$ be a non-zero derivation of $A$. Let $R=A[z;\delta]$ and let $B$ be the Poisson algebra $A[z]$ with the Poisson bracket $\{-,-\}_{\delta}$.
There is a homeomorphism between $\spec R$ and $\Pspec B$.
\end{theorem}
\begin{proof}
We have seen in \ref{twotypes} that there is a homeomorphism $\Gamma_1:\Pspec_1(B)\rightarrow \spec_1(R)$ such that $\Gamma_1(P)/JR=\kappa(P/JB)$ for all
$P\in \Pspec_1(B)$. We aim to extend this to a homeomorphism $\Gamma:\Pspec(B)\rightarrow \spec(R)$.

By Lemma \ref{PcapA},
every element of $\Pspec_2 B$ has the form $QB$ for a $\delta$-prime ideal $Q$ of $A$ such that $J\not\subseteq Q$
and, by Lemma \ref{PcapAore},
every element of $\spec_2 R$ has the form $QR$ for such an ideal $Q$. Define $\Gamma_2:\Pspec_2 B\rightarrow \spec_2 R$ by setting $\Gamma_2(QB)=QR$. Then $\Gamma_2$ is bijective and $\Gamma_2$ and $\Gamma_2^{-1}$ preserve inclusions. Combine $\Gamma_1$ and $\Gamma_2$ by defining a bijection $\Gamma:\Pspec B\rightarrow \spec R$ by
\[\Gamma(P)=\begin{cases}\Gamma_1(P)\text{ if }P\in \Pspec_1(B),\\ \Gamma_2(P)\text{ if }P\in \Pspec_2(B)
\end{cases}.\]

Inclusions within $\Pspec_1 B$ and $\spec_1R$ and within $\Pspec_2 B$ and $\spec_2R$ are preserved by $\Gamma$ and $\Gamma^{-1}$.  There are no inclusions $P^\prime\subseteq P$ with $P^\prime\in \Pspec_1 B$ and $P\in \Pspec_2 B$ or
with $P^\prime\in \spec_1 R$ and $P\in \spec_2 R$. Let $P^\prime=QB\in \Pspec_2 B$ and  $P\in \Pspec_1 B$.
Then \begin{eqnarray*}
QB\subseteq P&\Leftrightarrow&\frac{QB+JB}{JB}\subseteq\frac{P}{JB}\\
&\Leftrightarrow&\kappa\left(\frac{QB+JB}{JB}\right)\subseteq\kappa\left(\frac{P}{JB}\right)=\frac{\Gamma(P)}{JR}\\
&\Leftrightarrow&\frac{QR+JR}{JR}\subseteq\frac{\Gamma(P)}{JR}\\
&\Leftrightarrow&\Gamma(QB)=QR\subseteq\Gamma(P).
\end{eqnarray*}
Thus both $\Gamma$ and $\Gamma^{-1}$ preserve inclusions. By \cite[Lemma 9.4]{goodsemiclass}, $\Gamma$ is a homeomorphism.
\end{proof}

\begin{remark}\label{fds}
For many affine algebras, particularly enveloping algebras and quantum algebras, there are prime ideals that are not completely prime and there is an established homeomorphism
between the completely prime part of the spectrum of a deformation and the Poisson prime spectrum of a corresponding semiclassical limit. Some such algebras are discussed in \cite{dajfdsPm}, where a common theme is that the incompletely prime
ideals are annihilators of finite-dimensional simple modules of dimension $d>1$ and it is such a module, rather than its annihilator, that is reflected on the Poisson side, through a $d$-dimensional simple Poisson module. In the context of this paper, this issue is not present on either side. On the Ore side, Sigurdsson \cite{sig} shows that all prime ideals of $A[z;\delta]$ are completely prime. On the Poisson side, by \cite[Theorem 1]{dajfdsPm}, a $d$-dimensional simple Poisson module
over an affine Poisson algebra corresponds to a $d$-dimensional simple Lie module for the Lie algebra $M/M^2$ for some maximal Poisson ideal $M$.
In the context of the present paper, $M/M^2$ is always soluble and, by \cite[Corollary 1.3.13]{dix}, every finite-dimensional simple Lie module for  $M/M^2$ has dimension one.
\end{remark}

\section{Primitivity}
The purpose of this section is to show that, for a commutative affine domain $A$ with derivation $\delta$, the Ore extension  $A[z;\delta]$ is primitive if and only if $A[z]$ is Poisson primitive, for the Poisson bracket $\{-,-\}_\delta$,  and that, under the homeomorphism $\Gamma$ of Theorem~\ref{homeo}, Poisson primitive ideals of $A[z]$ correspond to primitive ideals of $A[z;\delta]$.

It follows from \cite[Theorems 1,2]{primore}, where $A$ is not necessarily affine, that if $\delta\neq 0$ and $A$ is either $\delta$-primitive or $\delta$-G then $A[z;\delta]$ is primitive. The converse, in the Noetherian case, was established in \cite[Theorem 3.7]{gwpaper2}. The logical independence, in the general case, of the two conditions, $\delta$-primitive and $\delta$-G, was shown in \cite{primore} by means of the non-affine examples $A=\C[[y]]$ with $\delta=y\d/dy$, which is $\delta$-G but not $\delta$-primitive, and $A=\C(t)[y]$ with $\delta=t\d/dt+y\d/dy$, which is $\delta$-primitive but not $\delta$-G.
If $A$ is affine and $\delta$-G
then $A$ is $\delta$-primitive by \cite[Proposition 1.2]{gooddm}. It would be interesting to know whether there is an affine $\delta$-primitive $\C$-algebra $A$ which is not $\delta$-G. Such an example
would have consequences for the Poisson Dixmier-Moeglin equivalence as it would give rise to a Poisson bracket on $A[z]$ for which $0$ is Poisson primitive, and hence Poisson rational, but not locally closed.

In the Poisson setting we have analogues of \cite[Theorems 1,2]{primore}.
\begin{theorem}\label{prim} Let $\delta$ be a non-zero derivation of a commutative $\C$-algebra $A$. Then $A[z]$, with the Poisson bracket $\{-,-\}_\delta$, is Poisson primitive if $A$ is $\delta$-primitive or $\delta$-G.
\end{theorem}
\begin{proof}
Suppose that $A$ is $\delta$-G and let $I\neq 0$ be the intersection of the non-zero $\delta$-prime ideals of $A$. As $A$ is a domain, the Jacobson radical $\Jac(A[z])=0$, for example by \cite[Theorem 4]{jacobsonstructure}. Suppose that $A[z]$ is not Poisson primitive and let $M$ be a maximal ideal of $A[z]$.  Then $\mathcal{P}(M)\neq 0$ and, by Lemma~\ref{PcapA}(i), $\mathcal{P}(M)\cap A$ is a non-zero $\delta$-prime ideal of $A$. Therefore $I\subseteq M$ for all maximal ideals $M$ of $A[z]$ so $I\subseteq \Jac(A[z])=0$.
This contradiction shows that $A[z]$ is Poisson primitive.

Suppose that $A$ is $\delta$-primitive and let $M$ be a maximal ideal of $A$ containing no non-zero $\delta$-ideal of $A$. Let $N$ be any maximal ideal of $A[z]$ containing $M$. Thus $N\cap A=M$. Let $P=\mathcal{P}(N)$. Then
$P=0$ otherwise, by Lemma~\ref{PcapA}, $P\cap A$ is a non-zero $\delta$-ideal of $A$ contained in $M$. Thus
$A[z]$ is Poisson primitive.
\end{proof}
It would be interesting to know whether the converse is true in the Noetherian case. As the next result shows, it is true in the affine case.
\begin{theorem}\label{affprim} Let $\delta$ be a non-zero derivation of a commutative affine $\C$-algebra $A$. Then $A[z]$ is Poisson primitive if and only if $A$ is $\delta$-primitive.
\end{theorem}
\begin{proof}
Suppose that $A[z]$ is Poisson primitive and let $M$ be a maximal ideal of $A[z]$ containing no non-zero Poisson ideal of $A[z]$. Then
$M\cap A$ contains no non-zero $\delta$-ideal of $A$ for if $J$ is a non-zero $\delta$-ideal of $A$ contained in $M\cap A$ then, by Lemma~\ref{QB}(ii),
$JA[z]$ is a non-zero Poisson ideal of $A[z]$ contained in $M$. But, by \cite[Theorem 27]{kap}, $A/(M\cap A)$ is a G-domain. As $A$ is affine, it is a Hilbert ring, by \cite[Theorem 31]{kap}, so $M\cap A$ is a maximal ideal of $A$. Thus $A$ is $\delta$-primitive. The converse holds by Theorem~\ref{prim}.
\end{proof}
\begin{cor}\label{BprimRprim} Let $\delta$ be a non-zero derivation of a commutative affine $\C$-algebra $A$. Then $A[z]$, with the Poisson bracket $\{-,-\}_\delta$, is Poisson primitive if and only if $A[z;\delta]$ is primitive.
\end{cor}
\begin{proof} As we observed above, \cite[Proposition 1.2]{gooddm} tells us that, in the affine case, if $A$ is $\delta$-primitive then $A$ is $\delta$-G. The result is then immediate from Theorem~\ref{affprim} and \cite[Theorem 3.7]{gwpaper2}.
\end{proof}
\begin{cor}  Let $\delta$ be a non-zero derivation of a commutative affine $\C$-algebra $A$, let $B=A[z]$ with the Poisson bracket $\{-,-\}_\delta$ and let $R=A[z;\delta]$. \begin{enumerate}
\item In $\Pspec(B)$, the Poisson primitive ideals are the maximal elements of $\Pspec_1 B$, that is the Poisson ideals $P$ of $B$ such that $B/P\simeq \C$, and the ideals of the form $QB$ where $Q$ is a $\delta$-primitive ideal of $A$.
\item In $\spec(R)$, the primitive ideals are the maximal elements of $\spec_1 R$, that is the ideals $P$ of $B$ such that $R/P\simeq \C$, and the ideals of the form $QR$ where $Q$ is a $\delta$-primitive ideal
of $A$.
\item In the homeomorphism between $\Pspec(B)$ and $\spec(R)$, the Poisson primitive ideals of $B$ correspond to the primitive ideals of $R$.
\end{enumerate}
\end{cor}
\begin{proof}
(i) Let $P$ be a Poisson prime ideal of $B$. Suppose first that $P\in\Pspec_1 B$. Then the Poisson bracket on $B/P$ is $0$ so $P$ is Poisson primitive if and only if $P$ is maximal if and only if $B/P\simeq \C$.
Now suppose that $P$ in $\Pspec_2 B$. Then $P=QB$ for some $\delta$-prime ideal $Q$ of $A$ with $\delta(A)\not\subseteq Q$ and, by Corollary~\ref{affprim} applied to $A/Q$, $P$ is Poisson primitive if and only if $Q$ is $\delta$-primitive.

(ii) Let $P$ be a  prime ideal of $R$. Suppose first that $P\in\spec_1 B$. Then $R/P$ is commutative so $P$ is  primitive if and only if $P$ is maximal if and only if $B/P\simeq \C$.
Now suppose that $P\in\spec_2 R$. Then $P=QR$ for some $\delta$-prime ideal $Q$ of $A$ with $\delta(A)\not\subseteq Q$ and, by Corollaries~\ref{affprim} and \ref{BprimRprim} applied to $A/Q$, $P$ is primitive if and only if $Q$ is $\delta$-primitive.

(iii) This follows from (i) and (ii).
\end{proof}

\section{Examples in $\C[x,y,z]$} \label{f30}
Here we look at some examples where $A=\C[x,y]$, so that $B=\C[x,y,z]$,
the polynomial algebra in three indeterminates. For $w=x,y$ or $z$, we denote the derivation
$\partial/\partial w$ of $B$ by $\partial_w$ and, for $a\in B$, we write $a_w$ for $\partial_w(a)$ and $\grad a$ for the triple $(a_x,a_y,a_z)\in B^3$.
 Poisson brackets on $\C[x,y,z]$ are the subject of \cite{dajsqoh}. Any such bracket is determined by the triple $(f,g,h)\in B^3$ such that
 \[
 \{y,z\}=f,\quad \{z,x\}=g\;\text{ and }\{x,y\}=h.\]
A triple $F=(f,g,h)\in B^3$ is a {\it Poisson triple} if it does determine a Poisson bracket in this way.
By \cite[Proposition 1.17(1)]{dajsqoh}, a triple $F=(f,g,h)\in B^3$ is a Poisson triple  if and only if
$F.\curl F=0$. Similar results are true for the rational function field $Q(B)=\C(x,y,z)$ and the completion
$\widehat{B}$ of $B$ at any maximal ideal.

For any $a,b\in B$, there is a Poisson bracket on $B$ such that
\[\{y,z\}=ba_x,\quad \{z,x\}=ba_y\;\text{ and }\{x,y\}=ba_z.\]
We call such a bracket \emph{exact} if $b=1$ and \emph{m-exact} in general. A Poisson bracket on $B$ is
\emph{qm-exact}, respectively \emph{cm-exact}, if there exist $a, b\in Q(B)$, resp $a, b\in \widehat{B}$ for some maximal ideal of $B$, such that
\[\{y,z\}=ba_x\in B,\quad \{z,x\}=ba_y\in B\;\text{  and }\{x,y\}=ba_z\in B.\]
In \cite{dajsqoh}, it is shown that every Poisson bracket on $B$ is cm-exact and the Poisson spectrum
is determined for a qm-exact bracket with $a=st^{-1}$ and $b=t^2$, $s$ and $t$ being coprime elements of $B$. Taking $t=1$, this includes the exact brackets.

In the remainder of this section, we consider non-exact Poisson brackets on
$B=\C[x,y,z]$ that extend the zero Poisson bracket on $A=\C[x,y]$, that is, we
consider Poisson brackets on $B$ with $\{x,y\}=0$.

\begin{lemma}
Let $f,g\in B$ and let $F=(f,g,0)$. Then $F$ is a Poisson triple
if and only
if there exist $h\in B$ and $f_1,g_1\in A$ such
that $f=hf_1$ and $g=hg_1$.
\end{lemma}
\begin{proof}
Suppose that $F$ is a Poisson triple. If $g=0$ we can take $h=1$,  $f_1=f$ and $g_1=0$ so we may assume that $g\neq 0$. As $\curl((f,g,0))=(-g_z,f_z,g_x-f_y)$, we have  $fg_z=gf_z$. Hence
$\partial_z(f/g)=0$ and
$pf=qg$ for some $p,q\in A$. Let $h$ be the highest common factor of $f$ and $g$ in $B$ and let $f_1, g_1\in B$ be such that $f=hf_1$ and $g=hg_1$. Then $pf_1=qg_1$. If $f_1\notin A$ then $f_1$ has an irreducible factor $u$ in $B\backslash A$ and, as $q\in A$, $u$ must divide $g_1$, contradicting the choice of $h$. Thus $f_1\in A$ and similarly
$g_1\in A$.

Conversely, suppose that $F=(hf_1,hg_1,0)$ where $h\in B$ and $f_1,g_1\in A$. Then $\curl F$ has the form
$(-g_1h_z,f_1h_z,\ell)$, where $\ell\in B$, so $F.\curl F=0$ and hence $F$ is a Poisson triple.
\end{proof}

The Poisson prime ideals of $B$ for a Poisson triple $F=(hf_1,hg_1,0)$ are the prime ideals containing
$h$ and the Poisson primes for the Poisson triple $(f_1,g_1,0)$ so it suffices to consider the case where $f,g\in A$.
Thus $\ham
x=-g\partial_z$, $\ham y=f\partial_z$ and $\ham
z=g\partial_x-f\partial_y$. Also $\ham z(A)\subseteq A$, so
that $\ham z$ restricts to a derivation of $A$
and the results of Section~\ref{poissonore} apply with $\delta$ being the restriction to $A$ of
$g\partial_x-f\partial_y$.

If $a\in A$ then the corresponding exact bracket on $B$
has $\{y,z\}=a_x$, $\{z,x\}=a_y$ and $\{x,y\}=0$. The following theorem is a special case of \cite[Theorem 3.8]{dajsqoh}.

\begin{theorem}
\label{Pspecexact} Let $a\in A\backslash\{0\}$. The Poisson prime ideals of $B$
under the exact bracket determined by $a$ are $0$, the residually null Poisson prime ideals
and the height one prime ideals $uA$, where $u$ is an irreducible factor of $a-\lambda$ for
some $\lambda\in \C$ such that $a-\lambda$ is a non-zero non-unit.
\end{theorem}

Combining this with Theorem~\ref{homeo} and its proof, we obtain the following corollary.
\begin{cor}
Let $a\in A\backslash\{0\}$, let $\delta$ be the derivation of $A$ such that $\delta(x)=a_y$ and $\delta(y)=-a_x$
and let $R=A[x;\delta]$. Let $J=a_yA+a_xA$. Then
\begin{enumerate}
\item $JR$ is an ideal of $R$ and $R/JR\simeq (A/J)[z]$.
\item The prime ideals of $R$
under the exact bracket determined by $a$ are $0$, the height one prime ideals $uR$, where $u$ is an irreducible factor of $a-\lambda$ for
some $\lambda\in \C$ such that $a-\lambda$ is a non-zero non-unit and the prime ideals of the form $\pi^{-1}(Q)$ where
$Q$ is a prime ideal of $(A/J)[z]$ and $\pi$ is the composition $R\twoheadrightarrow R/JR\simeq (A/J)[z]$.
\end{enumerate}
\end{cor}

\begin{example}
Let $a=x^2+y^2$. Then $\{z,x\}=2y, \{y,z\}=2x$ and $\{x,y\}=0$. The residually null Poisson prime ideals of $B$ are
$xB+yB$ and the maximal ideals that contain it. The other Poisson prime ideals of $B$
are $0$,
the height one prime ideals $(x+iy)A$, $(x-iy)A$ and $(x^2+y^2-\lambda)A$, where $\lambda\in \C^*$.
Note that those of the form $(x^2+y^2-\lambda)A$ are maximal Poisson ideals.

If $\delta=2y\partial_x-2x\partial_y$, so that $\delta(x)=2y$ and $\delta(y)=-2x$ and $R=A[z;\delta]$ then the prime spectrum of $R$ consists of $0$,
the height one prime ideals $(x+iy)R$, $(x-iy)R$ and $(x^2+y^2-\lambda)R$, where $\lambda\in \C^*$,
$xR+yR$ and $xR+yR+(z-\mu)R$, where $\mu\in \C$. For each $\lambda\in \C^*$, the algebra $R/(x^2+y^2-\lambda)R$ is simple.
\end{example}

In the remainder of the paper we consider non-exact Poisson brackets on $B$, beginning with some for which $B$ is Poisson simple. The following
result of Shamsuddin, for which a proof may be found in
\cite[Proposition 3.2]{Cout}, is useful in identifying Poisson
brackets for which $B$, or a localization of $B$, is Poisson simple.

\begin{prop}
Let $C$ be a commutative domain and let $g=at+b\in C[t]$, where
$a,b\in C$. Suppose that there exists a derivation $\delta$ of
$C[t]$ such that $\delta(C)\subseteq C$, $C$ is $\delta|_C$-simple,
$\delta(t)=g$ and,
 for all $r\in C$, $\delta(r)\neq ar+b$.
Then $C[t]$ is $\delta$-simple. \label{shamsim}
\end{prop}

\begin{examples}
\label{bergmanetal}
In the case where $A=\C[x,y]$ and $B=\C[x,y,z]$,
there are many known examples of simple derivations
$\delta=g\partial_x-f\partial_y$ of $A$. 
For all of these, $B$ is Poisson simple for the
Poisson bracket determined by the triple $(f,g,0)$.
In many of these examples $g=1$ so that
\begin{equation}\label{g1}
\{x,y\}=0,\quad \{z,x\}=1\;\text{ and }\{y,z\}=f.\end{equation}
In the best known example which is due to, but not published by, Bergman and is documented in \cite[\S6]{Cout2}, $f=-(1+xy)$.
The simplicity of $\delta$ follows easily from
Proposition~\ref{shamsim}, with $R=\C[x]$ and $t=y$ and the same argument works for $f=-(1+\lambda xy)$, $\lambda\in\C^*$. When $\lambda=-1$ and $x, y$ and $z$ are written $-x_1, x_3$ and $x_2$ respectively, this gives the first published
example, due to Farkas \cite[Example following Lemma 15]{farkas2},
of a Poisson bracket on $B$ for which $B$ is Poisson simple.

Other examples of polynomials $f\in A$ for which $B$ is Poisson simple under the Poisson bracket in \eqref{g1} include:
 \begin{enumerate}
\item
 $f=p(x)-y^2$, where $p(x)\in \C[x]$ has odd degree. See \cite[Theorem 6.2]{nowetal}.
\item
$f=-(y^m+ax^n)$, where $m,n\in \N$, $m\geq 2$ and $a\in \C\backslash\{0\}$. See \cite[Theorem 1]{havran} which generalised an earlier result \cite{nowicki}, for the case $n=1$.
\end{enumerate}
\end{examples}

\begin{example}\label{ox}
In contrast to the examples in Examples~\ref{bergmanetal}, $B$ is also Poisson simple for the Poisson bracket such that
\[
\{x,y\}=0,\quad \{z,x\}=y^3\;\text{ and }\{y,z\}=xy-1,\]
which has the property that, for all $b\in B$, $\{z,b\}$, $\{x,b\}$ and $\{y,b\}$ are not units.
Clearly $\{x,b\}=-y^3\partial_z(b)$ and $\{y,b\}=(xy-1)\partial_z(b)$ are never units. For $\{z,(\sum a_iz^i)\}=\sum \{z,a_i\}z^i\}$ to be a unit it is necessary that $\{z,a_i\}$ is a unit and it is shown in \cite{dajox} that if $a\in A$ then $\{z,a\}=\delta(a)$ is not a unit.
\end{example}

\begin{remark}
The examples in \ref{ox} and \ref{bergmanetal}(i) have analogues in the polynomial algebra $\C[x_1,x_2,\ldots,x_n]$ when $n>3$. In these $\{x_i,x_j\}=0$ for $1\leq i,j\leq n-1$ and $\ham z$ is a simple derivation of $\C[x_1,x_2,\ldots,x_{n-1}]$.
See \cite[\S3]{dajox} and \cite[\S9]{nowetal} for details from the point of view of Ore extensions.
\end{remark}

\begin{example}
Coutinho \cite{Cout2,Cout3} has used the theory of foliations to make a substantial contribution to the understanding of the simple derivations of $\C[x,y]$. Let $A_2$ be the subspace of $\C[x,y]$ consisting of polynomials of total degree at most $2$ and let $\mathcal{U}_2$ be the set of unimodular rows $(a,b)$
where $a, b\in A_2\times A_2$ are such that at least one of $a$ and $b$ has total degree $2$. In \cite{Cout3} it is shown that the closure $\overline{\mathcal{U}_2}$ in $A_2\times A_2$ has four irreducible components $\mathcal{P}_i$, $1\leq i\leq 4$ and examples of simple derivations from each component are given. For the first two types, the class of examples is dense in $\mathcal{P}_i$. Below we give the details of examples of four corresponding types of Poisson bracket on $B$ for which
$B$ is Poisson simple. Full details, presented from the point of view of derivations of $\C[x,y]$, can be found in \cite{Cout3}.

{\bf Type 1}, $\mathcal{P}_1$: let  $a,b,c\in \Q[i]\backslash{0}$, with $a\neq 1$ be such that the quadratic  polynomial $y^2+bx^2+cxy$ is irreducible over $\Q[i]$. Then, by \cite[Proposition 4.1]{Cout3} and Corollary~\ref{simple}, $\C[x,y,z]$ is Poisson simple for the Poisson bracket such that
\[\{x,y\}=0,\quad \{y,z\}=c(xy+a)+bx^2\text{ and }\{z,x\}=xy+a.\]

{\bf Type 2}, $\mathcal{P}_2$: let  $\beta\in \Q[i][x,y]$, be homogeneous of degree $2$ and irreducible over  $\Q[i]$. Then, by \cite[Proposition 5.4]{Cout3} and Corollary~\ref{simple}, $\C[x,y,z]$ is Poisson simple for the Poisson bracket such that
\[\{x,y\}=0,\quad \{y,z\}=-\beta\text{ and }\{z,x\}=1.\]

{\bf Type 3}, $\mathcal{P}_3$: by \cite[Proposition 6.1]{Cout3} and Corollary~\ref{simple}, $\C[x,y,z]$ is Poisson simple for the Poisson bracket such that
\[\{x,y\}=0,\quad \{y,z\}=-x\text{ and }\{z,x\}=xy+1.\]

{\bf Type 4}, $\mathcal{P}_4$: in Examples~\ref{bergmanetal}(i), take $p(x)=\rho x$ where $\rho\in \C\backslash\{0\}$.

For discussion of some classes of simple derivations $\delta=g\partial_x-f\partial_y$ of $A$ where the degrees of $f$ and $g$ may be greater than $2$, see \cite[Corollary 4.3, Theorems 4.4 and 5.5 and Proposition 6.2]{Cout2}.

\end{example}

\begin{example} Let $f=-1$ and $g=x$, so that $\delta(x)=x$ and $\delta(y)=1$ and
 the Poisson bracket on $A$ is such
that $\{y,z\}=-1, \{z,x\}=x$ and $\{x,y\}=0$. The Poisson triple
here is the cm-exact triple $x\grad(y-\log x)$. It is clear that
$xB$ is  a Poisson prime ideal and that $x\notin \Pz (B)$. Applying Proposition~\ref{shamsim} with $C=C[x^{\pm 1}]$, $a=0$, $b=1$ and $\delta|_C=xd/dx$, it is  easy  to see
that  $xA$ is the only non-zero $\delta$-prime ideal
of $A$. As $\delta(A)\nsubseteq xA$ it follows from Theorem~\ref{PcapA}
that $\Pspec(A)=\{0,xA\}$. By Theorem~\ref{homeo}, if $R=A[z;\delta]$
then $\spec R=\{0,xR\}$.
\end{example}

\begin{example}\label{new}
Let $M=xB+yB$ and $N=xA+yA$ and suppose that $f,g\in A$ are such
that if
$\delta=g\partial_x-f\partial_y$ then $N$ is the unique non-zero $\delta$-prime ideal of $A$, in other words, there are no height one prime ideals invariant under $\delta$ and $N$ is the only maximal ideal of $A$ invariant under $\delta$. Then $f=-\delta(y)\in N$, $g=\delta(x)\in N$ and $\delta(A)\subseteq N$. By
Theorems~\ref{PcapA} and \ref{homeo},
\[\Pspec B=\{0,xB+yB\}\cup \{xB+yB+(z-\alpha)B: \alpha\in \C\}\] and, if $R=C[z;\delta]$,
\[\spec R=\{0,xR+yR\}\cup \{xR+yR+(z-\alpha)R: \alpha\in \C\}.\]
Note that $\Pspec_2 B=\{0\}$ and $\spec_2 R=\{0\}$. In other words, there is no proper Poisson prime homomorphic image of $B$ with a non-zero Poisson bracket and every proper prime homomorphic image of $R$ is commutative. However if $j$ is such that $f\notin N^j$ or $g\notin N^j$ then $B/(N^jB)$ is a proper Poisson homomorphic image of $B$ with a non-zero Poisson bracket and $R/(N^jR)$ is a noncommutative proper homomorphic image of $R$. Such a $j$ must exist as $f$ and $g$ must be non-zero and $\cap_{j\geq 1}N^j=0$.

Goodearl and Warfield \cite[p. 61]{gwpaper1} specify such an example
with $f=-(x^2+y^2)$ and $g=x+y$. Although condition on the base field in
\cite{gwpaper1} is satisfied by $\R$ but not by $\C$, the conclusion
is also valid for $\C$. The details of this example were omitted from \cite{gwpaper1} as the proof was \lq exceedingly tedious\rq. Interest was expressed in any similar
example with a short proof. Here, subject to the reader's
interpretation of the word \lq short\rq, we present such an example.

Let $f=-x(1+xy)$ and $g=y$ so that,
$\delta(y)=x(1+xy)$ and $\delta(x)=y$. Let $^\prime$ denote differentiation with respect to $x$. Clearly $N$
is the unique maximal ideal of $A$ invariant under $\delta$. Let
$Q\neq N$ be a non-zero $\delta$-prime ideal of $A$.
Then $Q$ has height one and is principal, $Q=qA$, say, with $0\neq
q=\sum_{i=0}^n q_i(x)y^i$, each $q_i(x)\in \C[x]$, $q_n(x)\neq 0$
and, as $\delta(q_0(x))=yq_0^\prime(x)$, $n>0$. Let $h\in A$ be such that $\delta(q)=hq$. Note that, for
$0\leq i\leq n$,
\[\delta(q_i(x)y^i)=q_i^\prime(x)y^{i+1}+ix^2q_i(x)y^i+ixq_i(x)y^{i-1}.\]
Note also that $q_n^\prime(x)\in q_n(x)\C[x]$ whence $q_n^\prime(x)=0$ and $q_n(x)\in\C^*$.
Therefore $\deg_y(\delta(q))\leq n$ so $h=h(x)\in \C[x]$.
Comparing coefficients of $y^i$, $0\leq i\leq n$, in the equation $\delta(q(x))=h(x)q(x)$, we obtain
\begin{equation}\label{recurg}
(i+1)xq_{i+1}(x)=(h(x)-ix^2)q_i(x)-q_{i-1}^\prime(x),
\end{equation}
where $q_{-1}(x)=0=q_{n+1}(x)$. Note that $q_0(x)\neq 0$, otherwise
$q_i(x)=0$ for all $i$.
For $i\geq
0$, let $d_i=\deg(q_i(x))$, let $d=d_0$ and let $e_i=\deg (h(x)-ix^2)$.
By \eqref{recurg} with
$i=0$, $d_1=e_0+d_0-1$.
It follows from \eqref{recurg} that
\begin{equation}\label{growth}
\text{if } d_i+e_i>d_{i-1}-1 \text{ then }d_{i+1}=d_i+e_i-1.
 \end{equation}
In the following five situations, \eqref{growth} can be used to show, inductively, that the sequence $\{d_i\}$ is eventually strictly increasing. Hence these cases can be excluded.

(i) If $h(x)=0$ then $e_i=2$, when $i>0$,  $q_1(x)=0$ and $d_i=d-4+i$ whenever $i>1$.

(ii) If $h(x)$ has degree $r=0$ or $1$ then $e_0=r$, $e_i=2$ when $i>0$, $d_1=d+r-1$ and $d_i=d-r-2+i$ when $i>1$.

(iii) If $h(x)$ has degree $r\geq 3$ or $h(x)=ax^2+bx+c$ has degree $r=2$ and $a\notin \N$ then $d_i=d+i(r-1)$ for $i>0$.

(iv) If $h(x)=ax^2+bx+c$ has degree $2$, $a\in \N$ and $b\neq 0$
then $d_i=d+i$ for $0\leq i\leq a$, $d_{a+1}=d+a$ and $d_{a+j}=d+a+j-1$ for $j\geq 2$.

(v) If $h(x)=ax^2+c$ has degree $2$, $a\in \N$ and $c\neq 0$ then $d_i=d+i$ for $0\leq i\leq a$, $d_{a+1}=d+a-1$ and $d_{a+j}=d+a+j-2$ for $j\geq 2$.

This leaves only the case $h(x)=ax^2$, $a\in \N$, in which we need to keep track of leading coefficients as well as degrees.
Let $\alpha$ be the leading coefficient of $q_0(x)$. By repeated use of \eqref{recurg}, the leading
coefficient of $q_i(x)$ is $\bigl(
\begin{smallmatrix} a\\ i
\end{smallmatrix} \bigr)\alpha$ for $0\leq i\leq a$.  In particular, the leading
coefficients of $q_{a-1}(x)$ and $q_a(x)$ are $a\alpha$ and $\alpha$ respectively. By \eqref{growth}, $d_i=d+i$  for $0\leq i\leq a$.

By \eqref{recurg}, with $i=a$, $(a+1)xq_{a+1}(x)=-q_{a-1}^\prime(x)$ so
$d_{a+1}=d+a-3$ and the leading coefficient in $q_{a+1}(x)$ is
$-(d+a-1)a\alpha/(a+1)$.

From \eqref{recurg}, with $i=a+1$, we see that
$d_{a+2}\leq d+a-2$ and that the coefficient of $x^{d+a-2}$ in
$q_{a+2}(x)$ is $-(d+2a)\alpha/((a+1)(a+2))\neq 0$.
Therefore $d_{a+2}=d+a-2>d_{a+1}-e_{a+2}-1$. It now follows, inductively,
that $d_{a+j}=d+a+j-4$ for all $j\geq 3$, which is
impossible. This completes the proof that $\Pspec$ and $\spec R$
are as stated above.
\end{example}

\begin{example}
\label{GWJ}
 Here we consider the Poisson bracket on $B$ arising from
\cite[Example 2.15]{gwpaper1}, where
$\delta=2y\partial_x+(y^2+x)\partial_y$ so that $\{y,z\}=-(y^2+x)$,
$\{z,x\}=2y$ and $\{x,y\}=0$. In \cite{gwpaper1}, it is shown that the
only non-zero $\delta$-prime ideals of $A$ are the
maximal ideal $M:=xA+yA$ and the height one prime $Q:=(y^2+x+1)A$. Note that $Q\nsubseteq M$ and
that $\delta(A)\subseteq M$ but $\delta(A)\nsubseteq Q$. By Theorem~\ref{PcapA},
\[\Pspec B=\{0,(y^2+x+1)B,
xB+yB\}\cup \{xB+yB+(z-\alpha)B: \alpha\in \C\}.\] If $R=A[z;\delta]$ then
\[\spec R=\{0,
(y^2+x+1)R, xR+yR\}\cup \{xR+yR+(z-\alpha)R: \alpha\in \C\}.\]
Note that $\Pspec_2 B=\{0,(y^2+x+1)B\}$ and $\spec_2 R=\{0,(y^2+x+1)R\}$.
In contrast to Example~\ref{new}, there is a unique non-zero Poisson prime ideal that is not residually null.
\end{example}

\begin{remark}
In both Examples~\ref{new} and \ref{GWJ}, the Poisson algebra $B$ has a Poisson prime ideal $P=xB+yB$ which has height two as a prime ideal but is minimal as a non-zero Poisson  prime ideal. In both cases $P$ is residually null. To obtain examples of this phenomenon in which $P$ is not residually null, pass to $B^\prime=B[u,v]=\C[x,y,z,u,v]$ with the Poisson bracket such that $\{u,b\}=\{v,b\}=0$ for all $b\in B$ and $\{u,v\}=1$. This is the tensor product, as Poisson algebras, of $B$ and a copy of the coordinate ring of the symplectic plane. Then $xB^\prime+yB^\prime$ again has height two as a prime ideal and is minimal as a non-zero Poisson prime ideal but it is not residually null, having factor isomorphic to $\C[z,u,v]$ with $\{u,v\}=1$ and $\{u,z\}=\{v,z\}=0$.
\end{remark}

\end{document}